\documentclass[reqno, 12pt]{amsart}
\usepackage{amsmath}
\usepackage{framed,comment,enumerate}
\usepackage[pdftex]{graphicx}
\usepackage{epstopdf}
\usepackage{xcolor, framed}
\usepackage{color}

\usepackage{amsmath, amsthm, amssymb, esint}
\usepackage{amsfonts}
\usepackage{bbm}
\usepackage[ansinew]{inputenc}
\usepackage[dvips]{epsfig}
\usepackage{graphicx}
\usepackage[english]{babel}
\usepackage{enumerate}
\usepackage{hyperref}
\usepackage{fbox}

\def\R {\mathbb{R}}
\def\N{\mathbb{N}}
\def\eps{\varepsilon}

\def\dH{{\rm dist}_{\mathcal{H}}}

\def\vem{\vspace{0.6em}}

\newtheorem{thm}{Theorem}[section]
\newtheorem{prop}[thm]{Proposition} 
\newtheorem{conj}[thm]{Conjecture}

\newtheorem{cor}[thm]{Corollary}
\newtheorem{lem}[thm]{Lemma}
\theoremstyle{definition}

\numberwithin{equation}{section}

\theoremstyle{remark}
\newtheorem{rem}[thm]{Remark}

\title[Global solutions with superquadratic growth]{Global solutions to the thin obstacle problem with superquadratic growth} 

\author{Xavier Fern\'andez-Real}
\address{Institute of Mathematics, \'Ecole Polytechnique F\'ed\'erale de Lausanne, Lausanne, Switzerland}
\email{xavier.fernandez-real@epfl.ch}

\author{Hui Yu}
\address{Department of Mathematics,	National University of Singapore, Singapore}
\email{huiyu@nus.edu.sg}

\keywords{Thin obstacle problem, Signorini problem, global solution, superquadratic growth}

\subjclass[2010]{35R35, 35A02}

\thanks{X. F. was supported by the Swiss National Science Foundation (SNF grant  PZ00P2\_208930),   by the Swiss State Secretariat for Education, Research and Innovation (SERI) under contract number MB22.00034, and by the AEI project PID2021-125021NA-I00 (Spain).}

\begin{document}

\begin{abstract}
We study rigidity/flexibility properties of global solutions to the thin obstacle problem. For solutions with bounded positive sets, we give a classification in terms of their expansions at infinity. For solutions with bounded contact sets, we show that the contact sets are highly flexible  and can approximate arbitrary compact sets.

These phenomena have no counterparts in the classical obstacle problem.
\end{abstract}

\maketitle
\section{Introduction}

The classification of \textit{global solutions}, namely, solutions in the entire space, is a central theme in geometry and analysis. In the setting of  the \textit{classical obstacle problem}, this question has been studied for over 90 years.  

It first arose in the context of null quadrature domains \cite{KM, Sh}. In three dimensions,  it was shown that compact contact sets of global solutions must be ellipsoids \cite{Div31, Lew79}. This result, still under the compactness assumption, was later extended to all dimensions \cite{DF86, FS86}. A short proof can be found in \cite{EW}. 

Without the compactness assumption, global solutions were first classified in the plane \cite{Sak82}, where it was shown that the contact set must either be an ellipsoid or a limit of ellipsoids.
Recently, a similar classification was achieved, first in dimensions $n \ge 6$ \cite{ESW23}, and then in all dimensions \cite{EFW22}.

\vem

 In this work,  we study global solutions to the \textit{thin obstacle problem}. When posed in $\R^{n+1}$, it takes the form: 
\begin{equation}
\label{eq:main_TOP}
\left\{
\begin{array}{rcll}
u & \ge& 0 & \quad\text{on}\quad \{x_{n+1} = 0\},\\
\Delta u & = & 0 & \quad\text{in}\quad \R^{n+1}\setminus \{x_{n+1} = 0, u = 0\},\\
\Delta  u &  \le & 0&  \quad\text{in}\quad \R^{n+1},\\
u(x', x_{n+1}) & =&  u(x', -x_{n+1})& \quad \text{in}\quad \R^{n+1},
\end{array}
\right.
\end{equation}
where we have denoted 
$$x = (x', x_{n+1}) \in  \R^n\times \R.$$

 In this setting, we say that \[
\Lambda(u) = (\{u = 0\}\cap\{x_{n+1} = 0\})\times  \{0\}
\]
 is the \emph{contact set}; 
 its boundary in the relative topology of $\R^n$  is the \emph{free boundary}, namely, 
\[
\Gamma(u) = \partial_{\R^n}(\{u >  0\}\cap\{x_{n+1} = 0\})\times  \{0\};
\]
 and the \emph{positivity set} is 
 \[
 (\{u > 0\}\cap\{x_{n+1} = 0\})\times  \{0\}.
 \]
 
The thin obstacle problem is a classical free boundary problem, originally studied by Signorini in the context of linear elasticity \cite{Sig33, Sig59, KO88}. The same equations arise in various settings, including biology, fluid mechanics, and finance \cite{DL76, Mer76, CT04, Ros18, Fer22}. In the last two decades, the problem has also been extensively investigated within the mathematical community; see, for example, \cite{Caf79, AC04, ACS08, GP09, FS18, CSV20, FJ21, FS23, SY22, SY21, FT23,  FS24} and references therein. For a general introduction to the topic, we refer to \cite{PSU12, Fer22}.

\vem

Unlike the classical case, the classification of global solutions to the thin obstacle problem remains incomplete. A major difference lies in the rate of growth of a solution at infinity. For the classical obstacle problem, solutions grow quadratically. For the thin obstacle problem, however, solutions can have different rates of growth.

Despite this, in \cite{ERW} the authors established a bijective correspondence between global solutions with compact contact sets and  their polynomial approximation at infinity. For solutions with quadratic growth, they showed that compact contact sets are convex. Still under the assumption of quadratic growth, Eberle and the second author proved that compact contact sets are ellipsoids \cite{EY}. 

The approach in \cite{EY} relies heavily on the approximation of the thin obstacle problem by the classical obstacle problem. This is only possible for solutions with quadratic growth. Addressing solutions solutions with general growth requires new techniques specific to the thin obstacle problem. 
 
 \vem 
 
In this work, we  study solutions with general a growth rate and observe behaviors that have no analogue in the classical obstacle problem. 

More precisely, we identify two distinct types of behavior that illustrate the richness of global solutions in the thin setting. First, we characterize the set of global solutions with bounded \emph{positivity set} (a phenomenon that cannot happen in the classical obstacle problem). Second, we show that even for compact sets, dropping the quadratic growth assumption allows highly flexible contact sets. In particular, this resolves a conjecture from \cite{ERW} on the nonconvexity of compact contact sets for superquadratic solutions.

\subsection{Global solutions with bounded positivity set}
We say that a  function $u: \R^{n+1}\to \R$ has \textit{polynomial growth} if there exists some $m\in \N$ such that 
\begin{equation}
\label{eq:poly_m}
\left\| \, \frac{u(x)}{1+|x|^m}\right\|_{L^\infty(\R^{n+1})} < +\infty.
\end{equation}
The smallest $m$ for which \eqref{eq:poly_m} holds is the \textit{order} of the polynomial growth of $u$. 

The main result in \cite{ERW} establishes the correspondence between global solutions to \eqref{eq:main_TOP} with polynomial growth and compact contact sets, and even (in $x_{n+1}$)  harmonic polynomials $p$ with bounded negative set on $\{x_{n+1} = 0\}$; see Section~\ref{sec:nonconvex} below. 

In this work, by contrast, we provide a correspondence between global solutions to \eqref{eq:main_TOP} with polynomial growth and bounded positivity set, and odd (in $x_{n+1}$)  polynomials $p$ that vanish on the thin space, and whose normal derivative $\partial_{n+1}p $ at $\{x_{n+1} = 0\}$ has a bounded negative set on the thin space. In particular, we show the existence of a broad family of global solutions with bounded positivity set.  Unlike the compact contact set case, this phenomenon is exclusive to the thin setting.

Let $\mathcal{P}_{n+1}$ denote the set of polynomials in $\R^{n+1}$. If we define 
\[
\mathcal{P}^o_{n+1} := \left\{
\begin{array}{l}
q \in \mathcal{P}_{n+1} : \Delta (x_{n+1} q) = 0, ~~ q\text{ is even in $x_{n+1}$}\\[2mm]
\hspace{2.6cm} \overline{\{  q(x', 0) < 0\}}~~\text{is compact}
\end{array}
\right\},
\]
then we have the following result:
\begin{thm}
\label{thm:main1}
For $n \ge 2$, let $u$ be a global solution to \eqref{eq:main_TOP} with polynomial growth.  Then, the positivity set $\{u > 0, x_{n+1} = 0\}$ is bounded if and only if 
\begin{equation}
\label{eq:cond_inf_odd}
|u(x) + |x_{n+1}| q(x)\, |\to 0\quad\text{as}\quad |x|\to \infty, 
\end{equation} 
for some $q\in \mathcal{P}^o_{n+1}$. 

Moreover, given any $q\in \mathcal{P}^o_{n+1}$, there is a unique $u$ solution to \eqref{eq:main_TOP} such that \eqref{eq:cond_inf_odd} holds. 
\end{thm}

We remark again that such type of behaviour has no analogue for the classical setting. 

\subsection{Flexibility of compact contact sets} 
We turn to global solutions with compact contact sets. It is conjectured in \cite[Remark 2]{ERW} that such sets need not be ellipsoids---or, in the context of their paper, convex---at least for solutions with superquadratic growth.\footnote{For solutions with exactly quadratic growth, such a statement would be false, as shown in \cite{EY}.}

It is not difficult to construct examples supporting this conjecture: one can produce solutions whose contact set is disconnected (see Lemma~\ref{lem:twoballs}), or even not simply connected (see Lemma~\ref{lem:disk}). In fact, as soon as we move to the lowest superquadratic order, namely, $m=4$, the contact set  can depart from convexity or ellipsoidal geometry. 

For instance, we prove the following: 

\begin{lem}
\label{lem:starshaped}
There exists a global solution to \eqref{eq:main_TOP} with polynomial growth \eqref{eq:poly_m} and $m = 4$ whose contact set is compact, nonconvex, and star-shaped. 
\end{lem} 

We achieve this through a more general result, showing that solutions to the thin obstacle problem with compact contact sets can have them to be arbitrarily close (in Hausdorff distance) to the zero level set of a polynomial.

\begin{prop}
\label{prop:polysets}
Let $n \ge 2$. Let $f\in \mathcal{P}_n$ with $f\ge 0$ and degree $m \ge 2$. Then:
\begin{enumerate}[(i)]
\item \label{it:i}
For any $\eps > 0$ there exists a global solution $u$ to \eqref{eq:main_TOP} with polynomial growth \eqref{eq:poly_m} of order $m$, whose contact $\Lambda(u)$ set satisfies
\[
\left(\{f = 0\}\times\{0\}\right)\cap B_1\subset \Lambda(u) \subset \{ (x', 0) \in \R^{n+1} : {\rm dist}(x', \{f = 0\})\le \eps\}\cap B_{R},
\]
for some constant $R>1$ depending only on $n$. 

\item For any $\eps > 0$ there exists a global solution $u$ to \eqref{eq:main_TOP} with polynomial growth \eqref{eq:poly_m} of order $m_\eps$ depending on $\eps$, whose contact $\Lambda(u)$ set satisfies
\[
\left(\{f = 0\}\times\{0\}\right)\cap B_1\subset \Lambda(u) \subset \{ (x', 0) : {\rm dist}(x', \{f= 0\})\le \eps\} \cap B_{1+\eps}.
\]
\end{enumerate}
\end{prop}
\begin{rem}
\label{rem:mhom}
In part \eqref{it:i}, if the polynomial $f$ is $m$-homogeneous, the global solution $u$ satisfying \eqref{eq:main_TOP_2} for some $p$ can be taken with $p = p_h - c_f$, for some $m$-homogeneous polynomial $p_h$, and a constant $c_f$. 
\end{rem}

As a consequence, we show that compact contact sets exhibit no rigidity at all  if one allows arbitrarily large growths.  They can approximate \emph{any} compact set $K$ in the following sense:

\begin{thm}
\label{thm:main2}
Let $n \ge 2$. Let $K\subset \{x_{n+1} = 0\}$ be any compact set. Then, for any $\eps > 0$ there exists a solution $u = u_\eps$ to \eqref{eq:main_TOP} with polynomial growth such that 
\[
K\subset \Lambda(u_\eps)\quad\text{and}\quad \dH(\Lambda(u_\eps), K) \le \eps,
\]
and 
\[
\dH(\Gamma(u_\eps), \partial_{\R^n} K) \le \eps,
\]
where $\dH$ denotes the Hausdorff distance between sets. 
\end{thm}

\begin{rem}
Equivalently, global solutions to the thin obstacle problem with \textit{analytic} obstacles exhibit the same flexibility. Previously, such results were only known in the setting of \(C^\infty\) obstacles; see \cite[Proposition 5.2]{FR21}.
\end{rem}

\subsection{Structure of the paper} The paper is structured as follows: 

We start with some preliminaries in Section~\ref{sec:prelim}. In Section~\ref{sec:mainthm} we prove Theorem~\ref{thm:main1}. In Section~\ref{sec:conj}, we prove a characterization of the contact set for cubic solutions, and pose a conjecture on the shape of positivity contact sets for solutions with cubic growth.  In Section~\ref{sec:nonconvex} we construct some examples of global solutions with nonconvex and compact contact set, and prove Lemma~\ref{lem:starshaped} and Proposition~\ref{prop:polysets}. Finally, in Section~\ref{sec:6}, we prove Theorem~\ref{thm:main2}.

\section{Preliminaries}
\label{sec:prelim}

Let us start by recalling and proving some preliminary considerations that will be useful later. 

\subsection{Notation} In this work, $x\in \R^{n+1}$ and we will denote $x = (x', x_{n+1})\in \R^n\times \R$. 

Likewise, we denote $B_r'(x_\circ')\subset \R^n$ the ball of radius $r > 0$ centered at $x_\circ'\in \R^n$ in dimension $n$, and $B_r(x_\circ)\subset \R^{n+1}$ the analogue in dimension $n+1$. 

Finally, it will be convenient to define the following one-sided normal derivative, for a given function $u:\R^{n+1}\to \R$, and a given point $x = (x', 0) \in \R^{n+1}\cap \{x_{n+1} = 0\}$:
\begin{equation}
\label{eq:norm_onesided}
\partial_{n+1}^+ u(x) = \lim_{x_{n+1} \to 0^+} \partial_{n+1} u(x', x_{n+1}). 
\end{equation}

\subsection{Classification of global solutions with compact contact set}

As already mentioned in the introduction, the main result in \cite{ERW} establishes the correspondence between global solutions to \eqref{eq:main_TOP} with polynomial growth and compact contact sets, and even (in $x_{n+1}$)  harmonic polynomials $p$ with bounded negative set on $\{x_{n+1} = 0\}$.

 That is, let us denote  
\[
\mathcal{P}^e_{n+1} := \left\{
\begin{array}{l}
 p \in \mathcal{P}_{n+1} : \Delta p = 0, ~~ p\text{ is even in $x_{n+1}$},\\[2mm]
\hspace{2.6cm} \overline{\{  p(x', 0) < 0\}}~~\text{is compact}
\end{array}
\right\}.
\]

Then, \cite[Theorem 1]{ERW} states that a global solution $u$ to \eqref{eq:main_TOP} with polynomial growth has compact contact set   if and only if 
\begin{equation}
\label{eq:main_TOP_2}
|u(x) - p(x)|\to 0\quad\text{as}\quad |x|\to \infty,\quad\text{for some $p\in \mathcal{P}^e_{n+1}$}. 
\end{equation}
Moreover, given any $p\in \mathcal{P}^e_{n+1}$ one can construct a unique $u$ solution to \eqref{eq:main_TOP} with $|u(x) - p(x)|\to 0$ as $|x|\to \infty$. 

More precisely, we have (see also the Appendix in \cite{EY}):
\begin{thm}[\cite{ERW}]
\label{thm:ERW}
Let $n \ge 2$, and let $u$ be a solution to \eqref{eq:main_TOP} with polynomial growth and  compact contact set. Then, there is a unique polynomial $p\in \mathcal{P}^e_{n+1}$ such that 
\begin{equation}
\label{eq:vp}
u(x) = u_p(x) = p(x) +v_p(x), 
\end{equation}
where $v_p(x)$ is the unique solution to  the thin obstacle problem with obstacle equal to $-p$: 
\begin{equation}
\label{eq:TOP_-p}
\left\{
\begin{array}{rcll}
v_p & \ge& -p & \quad\text{on}\quad \{x_{n+1} = 0\}\\
\Delta v_p & = & 0 & \quad\text{in}\quad \R^{n+1}\setminus \{x_{n+1} = 0, u = -p\}\\
\Delta v_p  &  \le & 0&  \quad\text{in}\quad \R^{n+1}\\
v_p(x', x_{n+1}) & =&  v_p(x', -x_{n+1})& \quad \text{in}\quad \R^{n+1}\\
v_p(x) & \to & 0& \quad\text{as}\quad |x|\to \infty. 
\end{array}
\right.
\end{equation}

Conversely, for any $p\in \mathcal{P}^e_{n+1}$ we have that \eqref{eq:vp} defines a solution to \eqref{eq:TOP_-p} with bounded contact set. 
\end{thm}

Notice that, in the previous statement, the contact set of $u$ is the same as the contact set of $v_p$, which is the set $(\{v_p = -p\}\cap \{x_{n+1} = 0\})\times \{0\}$.

It is then further shown that a global solution to \eqref{eq:main_TOP} with compact contact set and quadratic growth is necessarily convex (in the directions of the thin space), and in particular, the contact set is convex.  In \cite{EY} it is then proved that, in that case, the contact set must be an ellipse. It is finally conjectured that, such a convexity  is lost with higher-order growths (and in particular, it can fail at quartic growth already).

\subsection{Comparison principle} 
We also recall the following comparison principle for  solutions to the thin obstacle problem (see, for example, \cite[Lemma 5]{ERW}):
\begin{lem}
\label{lem:comp}
Let $U$ be a bounded domain, even in $x_{n+1}$, let $\varphi\in C(U\cap\{x_{n+1} = 0\})$, and let $v_1, v_2\in C(U)$ be two (viscosity) solutions to 
\begin{equation}
\label{eq:comp}
\left\{
\begin{array}{rcll}
v_i & \ge& \varphi & \quad\text{on}\quad U\cap \{x_{n+1} = 0\}\\
v_i(x', x_{n+1}) & =&  v_i(x', -x_{n+1})& \quad \text{in}\quad U\\
\Delta v_i & \le & 0 & \quad\text{in}\quad U\\
\Delta v_1 & = & 0&  \quad\text{in}\quad U\setminus\{x_{n+1} = 0, v_1 = \varphi\},
\end{array}
\right.
\end{equation}
for $i = 1, 2$, such that 
\[
v_1 \le v_2 \quad\text{on}\quad \partial U. 
\] 
Then,
\[
v_1 \le v_2\quad\text{in}\quad U. 
\]
\end{lem}

As a consequence of the previous result, we get the comparability of global solutions to the thin obstacle problem:
\begin{cor}
\label{cor:comp}
Let $\varphi\in C(\{x_{n+1} = 0\})$, and let $v_1, v_2\in C(\R^{n+1})$ be two global (viscosity) solutions to 
\begin{equation}
\label{eq:comp2}
\left\{
\begin{array}{rcll}
v_i & \ge& \varphi & \quad\text{on}\quad \R^{n+1}\cap \{x_{n+1} = 0\}\\
v_i(x', x_{n+1}) & =&  v_i(x', -x_{n+1})& \quad \text{in}\quad \R^{n+1}\\
\Delta v_i & \le & 0 & \quad\text{in}\quad \R^{n+1}\\
\Delta v_1 & = & 0&  \quad\text{in}\quad \R^{n+1}\setminus\{x_{n+1} = 0, v_1 = \varphi\},
\end{array}
\right.
\end{equation}
for $i = 1, 2$. Let us suppose, moreover, that 
\[
\liminf_{|x|\to \infty}\,  (v_2 - v_1)(x) \ge 0. 
\]
Then,
\[
v_2 \ge v_1\quad\text{in}\quad \R^{n+1}. 
\]
\end{cor}
\begin{proof}
By assumption, 
\[
\inf_{\R^{n+1}\setminus B_R} (v_2 - v_1)\ge - \omega_R\quad\text{for all}\quad R > 0,
\]
 for some $\omega_R \downarrow 0$ as $R\to \infty$. 
 
 We therefore can apply Lemma~\ref{lem:comp}  to the functions $v_1$ and $v_2 +\omega_R$, which satisfy \eqref{eq:comp} with $U = B_R$, to deduce 
 \[
 v_2+\omega_R \ge v_1 \quad\text{in}\quad B_R. 
 \]
 Letting $R\to \infty$ we obtain the desired result. 
\end{proof}

\section{Proof of Theorem ~\ref{thm:main1}}
\label{sec:mainthm}

Let us start by proving the result on the classification of global solutions with bounded positivity set. Before that, we state and prove the following Liouville-type result for exterior domains:

\begin{lem}
\label{lem:Liouv}
Let $n\ge 3$, and let $u:\R^n\to \R$ be a function with polynomial growth of order $m$ (recall \eqref{eq:poly_m}) such that 
\[
\Delta u = 0\quad\text{in}\quad \R^n\setminus B_1. 
\]
Then, there exists a polynomial $p$ of order $m$ such that 
\[
\begin{split}
|u(x) - p(x)| &= O(|x|^{2-n})\\
|\nabla u(x) - \nabla p(x)| &= O(|x|^{1-n}).
\end{split}
\]
Moreover, if $u$ is odd (resp. even) with respect to $x_1$, then $p$ is odd (resp. even) with respect to $x_1$ as well. 
\end{lem}
\begin{proof}
Let $\varphi\in C^\infty(\R_+)$ such that $\varphi\ge 0$, $\varphi(t) \equiv 1$ for $t\ge 3$, $\varphi \equiv 0$ for $t\in (0, 2)$. Then, take 
\[
\bar u(x) = \varphi(|x|) u(x)\quad\text{for}\quad x\in \R^n. 
\]

Since $\Delta u = 0$ in $\R^n\setminus B_1$, $u$ is smooth there, and $\bar u\in C^\infty(\R^n)$, with $\bar u \equiv  u$ in $\R^n\setminus B_3$ and $\bar u \equiv 0$ in $B_2$. Moreover, if $u$ is odd/even with respect to a coordinate, so is $\bar u$. 

Let 
\[
f(x) := \Delta \bar u(x), 
\]
so that $f\in C^\infty_c(B_3)$. If $\Gamma_n(x) $ denotes the fundamental solution to the Laplace equation in $\R^n$, we can define 
\[
w_f  := \Gamma_n * f,  
\]
which satisfies 
\[
\left\{
\begin{array}{rcll}
\Delta w_f & =& f&\quad\text{in}\quad \R^n,\\
|w_f(x) |& \le& C \|f\|_{L^\infty(B_3)}|x|^{2-n}&\quad\text{for all}\quad x\in \R^n,
\end{array}
\right.
\]
since $f$ is compactly supported and thanks to the decay of $\Gamma_n(x)$. Furthermore, $w_f$ has the same odd/even symmetries as $\bar u$ and $u$. 

Thus, $\bar u - w_f$ is a globally defined harmonic function with polynomial growth of order $m$. By the classical Liouville theorem, there exists a harmonic polynomial $p$ of degree $m$ such that $\bar u - w_f = p$, and therefore, 
\[
|\bar u(x) - p(x)| = O(|x|^{2-n}). 
\]
Since $u = \bar u$ in $\R^n\setminus B_3$ this proves the first equation, and by classical harmonic estimates applied in $B_{|x|/2}(x)$ we obtain the second equation. 
\end{proof}

We can now proceed with the proof of Theorem~\ref{thm:main1}:

\begin{proof}[Proof of Theorem~\ref{thm:main1}]
Let us divide the proof into four steps. 

We show first that, given   $q\in \mathcal{P}^o_{n+1}$, there is a unique solution to \eqref{eq:main_TOP} such that \eqref{eq:cond_inf_odd} holds and it has bounded positivity set $\{u > 0, x_{n+1} = 0\}$. In the last step, we show the converse implication.
\\[2mm]
\noindent {\bf Step 1}. 
The uniqueness of the solution follows from Corollary~\ref{cor:comp}: if there were two global solutions for which \eqref{eq:cond_inf_odd} holds for the same $q$, denoted $u_1$ and $u_2$, then by triangle inequality
\[ 
\|u_1 - u_2\|_{L^\infty(\R^{n+1}\setminus B_R)}  \to 0\quad\text{as}\quad R\to \infty,
\] 
and thanks to Corollary~\ref{cor:comp} (applied for both $(v_1, v_2) = (u_1, u_2)$ and $(v_1, v_2) = (u_2, u_1)$) we get $u_1 = u_2$. 
\\[2mm]
\noindent {\bf Step 2}.
Let us now prove the existence of a global solution, that will arise by compactness of the sequence $u_R$ as $R\to \infty$, where $u_R$ is the solution to 
\[
\left\{
\begin{array}{rcll}
u_R & \ge& 0 & \quad\text{on}\quad B_R\cap \{x_{n+1} = 0\}\\
\Delta u_R & = & 0 & \quad\text{in}\quad B_R\setminus \{x_{n+1} = 0, u_R = 0\}\\
\Delta  u_R &  \le & 0&  \quad\text{in}\quad B_R\\
u_R(x)& =&  -|x_{n+1}|q(x) & \quad \text{for}\quad x\in \partial B_R.
\end{array}
\right.
\]
 Indeed, since $\Delta u_R = \Delta (x_{n+1} q) = 0$ in $B_R^+ = B_R\cap \{x_{n+1} > 0\}$ and $u_R\ge 0$ on $B_R\cap \{x_{n+1} = 0\}$,  by maximum principle and even extension we have 
\[
-|x_{n+1}| q \le u_R\quad\text{in}\quad B_R.
\]

On the other hand, let us suppose that $\overline{\{q(x', 0) < 0\}}\subset B_\rho$, for some $\rho \ge 1$, and consider $\varphi\in C^\infty_c(B_{\rho+1})$ a radial test function with $\varphi \ge 0$ and $\varphi \equiv 1$ in $B_\rho$, and $\Phi$ to be a global solution to the thin obstacle problem with obstacle $\varphi$ vanishing at infinity (constructed, for example, by Perron's method\footnote{That is, $\Phi$ is the least supersolution above the obstacle on the thin space and nonnegative at infinity.}): 
\begin{equation}
\label{eq:Phi}
\left\{
\begin{array}{rcll}
\Phi & \ge& \varphi & \quad\text{on}\quad \{x_{n+1} = 0\}\\
\Delta \Phi & = & 0 & \quad\text{in}\quad \R^{n+1}\setminus \{x_{n+1} = 0, u = \varphi\}\\
\Delta \Phi &  \le & 0&  \quad\text{in}\quad \R^{n+1}\\
\Phi(x', x_{n+1}) & =&  \Phi(x', -x_{n+1})& \quad \text{in}\quad \R^{n+1}\\
\Phi(x) & \to & 0& \quad\text{as}\quad |x|\to \infty. 
\end{array}
\right.
\end{equation}

Then, by construction, $0\le \Phi \le 1$ in $\R^{n+1}$, $\partial^+_{n+1}\Phi \le 0$ on $\{x_{n+1} = 0\}$ (recall~\eqref{eq:norm_onesided}), and by Hopf's lemma, $\partial^+_{n+1}\Phi \le - c_\rho < 0$ on $\{x_{n+1} = 0\} \cap B_\rho$ for some $c_\rho > 0$. In particular, since we must have $q\ge -C_q>-\infty$ on $\{x_{n+1} = 0\}$ and $q\ge 0$ on $\{x_{n+1} = 0\}\setminus B_\rho$, there exists some $\kappa$ depending on $C_q$ and $\rho$ such that 
\[
-q(x) + \kappa\partial^+_{n+1}\Phi \le 0\quad\text{on}\quad \{x_{n+1} = 0\}. 
\]
(We can take $\kappa = {C_q}/{c_\rho}$.) This means that $\partial^+_{n+1}\left( -x_{n+1} q (x) +\kappa \Phi\right)\le 0$ on $\{x_{n+1} = 0\}$, it is harmonic on $\{x_{n+1} \neq 0\}$, and smooth on $\{x_{n+1} = 0\}$. Hence, $-|x_{n+1}| q (x) +\kappa \Phi$ is a supersolution, above $u_R$ on $\partial B_R$, and thus
\[
u_R\le -|x_{n+1}| q + \kappa \Phi\quad\text{in}\quad B_R. 
\]

In all, we have 
\[
-|x_{n+1}|q\le u_R \le -|x_{n+1}|q + \kappa\Phi \quad\text{in}\quad B_R
\]
for any $R\ge 1$. We can now take $R\to \infty$, and by compactness $u_R$ converges to a global solution to \eqref{eq:main_TOP} $u$, with $
-|x_{n+1}|q\le u \le -|x_{n+1}|q + \kappa\Phi 
$ in $\R^{n+1}$, and therefore, $|u+|x_{n+1}|q(x)|\to 0$ as $|x|\to \infty$ (since $\Phi(x) \to 0$ as $|x|\to \infty$). By the previous step, moreover, $u$ is unique. 
\\[2mm]
\noindent {\bf Step 3}.
Let us finish this first implication by showing that $u$ has compact positivity set on $\{x_{n+1} = 0\}$. Observe first that, since 
$\partial^+_{n+1} u(x', 0)  \le 0
$,
and $u\ge -|x_{n+1}|q$ in $\R^{n+1}$, if we had $u(x_\circ) = 0$ at some $x_\circ\in \{x_{n+1} = 0\}\cap \{q < 0\}$, then we would have 
\[
0\ge \partial^+_{n+1} u(x'_\circ, 0) \ge -q(x_\circ) > 0,
\]
a contradiction. Thus, $u > 0$ on $\{q < 0\}\cap \{x_{n+1} = 0\}$, and if $\{q < 0\}\cap \{x_{n+1} = 0\} \neq \varnothing$, then $u$ is not identically zero on the thin space. 

Let now $x_\circ\in \{x_{n+1} = 0\}$ with $|x_\circ|\ge R_\circ+1$, where $R_\circ>0$ depending only on $q$ is large enough so that $\kappa  \Phi  \le \frac{1}{2n}$   and $q \ge 1$ in $\R^{n+1}\setminus B_{R_\circ}$ (where $\Phi$ and $\kappa$ are defined in the previous step). Let us   consider a cylinder around $x_\circ$, 
\[
Q(x_\circ) := \{x\in \R^{n+1} : |x_{n+1}|< 1, ~~ |x'-x_\circ'|\le 1\}. 
\]

We will show that $u(x)\le \xi(x) := -|x_{n+1}|q(x) +|x_{n+1}|-\frac12 x_{n+1}^2+\frac{1}{2n}|x'-x_\circ'|^2$ for $x\in Q(x_\circ)$.  Indeed, first notice that since $q \ge 1$ in $Q(x_\circ)$, we have 
\[
\Delta\xi = (-2q(x) +2)\mathcal{H}^{n}(\{x_{n+1} = 0\})\le 0\quad\text{in}\quad Q(x_\circ),
\]
and $\xi$ is a supersolution. On the other hand, we have $u \le -|x_{n+1}|q+ \kappa\Phi \le -|x_{n+1}|q+\frac{1}{2n}\le\xi $ on $\partial Q(x_\circ)$. 
 
Since $u$ is a solution to the thin obstacle problem in $Q(x_\circ)$, the comparison principle, Lemma~\ref{lem:comp}, implies $u \le   \xi$. In particular, $0\le u(x_\circ) \le \xi(x_\circ) = 0$, and $x_\circ\notin\{u > 0\}\cap \{x_{n+1} = 0\}$. Since $x_\circ$ was arbitrary in $B_{R_\circ}\cap \{x_{n+1} = 0\}$, we deduce that $\R^{n+1}\setminus B_{R_\circ}\subset \{u = 0, x_{n+1} = 0\}$ for some $R_\circ > 0$, and $\{u > 0, x_{n+1} = 0\}$ is bounded. 
\\[2mm]
\noindent {\bf Step 4}. Let us finally show that if $u$ is a global solution with bounded positivity set, then \eqref{eq:cond_inf_odd} holds for some $q\in \mathcal{P}^o_{n+1}$.

Indeed, let $\{u > 0, x_{n+1} = 0\}\subset B_{\rho/2}\subset B_\rho$. Then, $u \equiv 0$ on $\{x_{n+1} = 0\}\setminus B_{\rho/2}$, and the odd extension 
\[
u^o(x) := \left\{\begin{array}{ll}
u(x) & \quad \text{if}\quad x_{n+1} \ge 0\\
-u(x) & \quad \text{if}\quad x_{n+1}< 0,
\end{array}
\right. 
\]
satisfies 
\[
\Delta u^o = 0\quad\text{in}\quad \R^{n+1}\setminus B_{\rho/2},
\]
and still has polynomial growth of order $m\in \N$, \eqref{eq:poly_m}. By Lemma~\ref{lem:Liouv}, there exists some harmonic polynomial $\bar q$, such that 
\[
\begin{split}
|u^o(x) - \bar q (x)| & = O(|x|^{1-n})\\
| \nabla u^o(x) - \nabla \bar q (x)| & = O(|x|^{-n}).
\end{split}
\]
Moreover, since $u^o$ is odd in $x_{n+1}$, so is $\bar q = x_{n+1} q$ (for some $q$ even in $x_{n+1}$ and of degree $m-1$), and we also have 
\[
\begin{split}
|u(x) +|x_{n+1}| q(x)|& = O({|x|^{1-n}})\\
|\partial_{n+1} u(x) +\partial_{n+1} (|x_{n+1}| q(x))|& = O({|x|^{-n}}),
\end{split}
\]
where $q$ is a polynomial even in $x_{n+1}$ and such that $x_{n+1} q$ is harmonic. Finally, if $\{q < 0, x_{n+1} = 0\}$ is not bounded, being $q$ a polynomial of degree at most $m-1$ on $\{x_{n+1} = 0\}$ we must have a sequence $x_k\in \{q < -1, x_{n+1} = 0\}$ with $|x_k|\to \infty$, and so by the previous estimate
\[
0\ge \partial^+_{n+1} u(x_k)\ge -q(x_k) - C|x_k|^{-n} \ge \frac12
\]
for $k$ large enough, reaching a contradiction.  Thus, $\{q < 0, x_{n+1} = 0\}$ is bounded, and $q\in \mathcal{P}^o_{n+1}$. 
\end{proof}

\section{A conjecture on cubic solutions}
\label{sec:conj}

Let us observe that, in the previous setting, we can give a characterization of solutions with compact positivity set when the   polynomials $q$ have a specific form. (See  Lemma~\ref{lem:nonconvex_starshaped} below for a counterpart for global solutions with compact contact set.)

\begin{lem}
\label{lem:nonconvex_starshaped2}
Let $n \ge 2$. Let $q\in \mathcal{P}^o_{n+1}$ be of the form $q = q^h - 1$ for some homogeneous polynomial $q^h$ of order $m\ge 2$. Then, the solution $u$ to \eqref{eq:main_TOP} given by Theorem~\ref{thm:main1} with $|u(x) +|x_{n+1}| q(x)\,  |\to 0$ as $|x|\to \infty$ has a bounded and star-shaped positivity set. 
\end{lem}
\begin{proof}
Let $u_r := r^{-m-1} u(r\, \cdot\,)$ for some $r \ge 1$. Then $u_r$ is a global solution to \eqref{eq:main_TOP} such that $| u_r + |x_{n+1}| q_r\, |\to 0$ as $|x|\to \infty$, where $q_r(x) = r^{-m}q(rx) = q^h(x) - r^{-m}\ge q(x)$, since $r\ge 1$. In particular, $u_r\ge u$ and we have 
\[
\{u_r = 0, x_{n+1} = 0\} = \frac{1}{r}\{u = 0, x_{n+1} = 0\} = \frac{1}{r}\Lambda(u) \subset \Lambda(u)\qquad\text{for all}\quad r \ge 1. 
\]
That is, $\{x_{n+1} = 0\}\setminus \Lambda(u)$ is star-shaped at the origin. 
\end{proof}

In particular, the previous result applies to the case of global cubic solutions with compact positivity set: in such a situation, the positivity set is necessarily star-shaped at some point. This is the only characterization of the positivity set for global cubic solutions available at the moment, under the assumption of having a bounded positivity set. We conjecture the following: 

 \begin{conj}
 Let $u$ be a global solution to the thin obstacle problem \eqref{eq:main_TOP} with cubic growth and a bounded positivity set on the thin space. Then, the positivity set on the thin space is convex, or more specifically, it is an ellipsoid. 
 \end{conj}

\section{A nonconvex compact contact set}
\label{sec:nonconvex}

In \cite[Remark 2]{ERW} it is conjectured that, for general growths, not all compact contact sets are convex.  Let us show its validity by means of the following simple example, obtained using Theorem~\ref{thm:ERW}: we immediately have that one can construct solutions with quartic growth where the contact set is not connected (and in particular, not convex).

\begin{lem}
\label{lem:twoballs}
There exists a global solution to \eqref{eq:main_TOP} with polynomial growth \eqref{eq:poly_m} and $m = 4$ whose contact set is compact and nonconvex.
\end{lem}
\begin{proof}
Let $\bar p\in \mathcal{P}_n$ be defined by 
\[
\bar p(x') = \bar p(x'', x_n) = 8|x''|^2+8(x_n^2-1)^2-1,
\]
and consider $p\in \mathcal{P}_{n+1}$ to be its harmonic extension towards $\R^{n+1}$. Let $v_p$ be the corresponding solution to \eqref{eq:TOP_-p} in Theorem~\ref{thm:ERW} (in particular, $v_p \ge 0$ in $\R^{n+1}$). Then, since $v_p$ is the least supersolution above the obstacle (alternatively, by Lemma~\ref{lem:comp} with $v_p$ and the constant function 1), and $\varnothing\neq \{-p> 0, x_{n+1} = 0\}\subset B_{1/2}(e_n) \cup B_{1/2}(-e_n)$,   we must have (also by symmetry in $x_n$)
\[
\{v_p = -p, x_{n+1} = 0\}\cap B_{1/2}(e_n) \neq \varnothing\quad\text{and}\quad \{v_p = -p, x_{n+1} = 0\}\cap B_{1/2}(-e_n)\neq\varnothing.
\] 
In all, the contact set of $v_p$ is disconnected and bounded, and in particular, it is compact and nonconvex. See Figure~\ref{fig:twoballs}.
\end{proof}

 \begin{figure}
 \includegraphics[scale = 0.6]{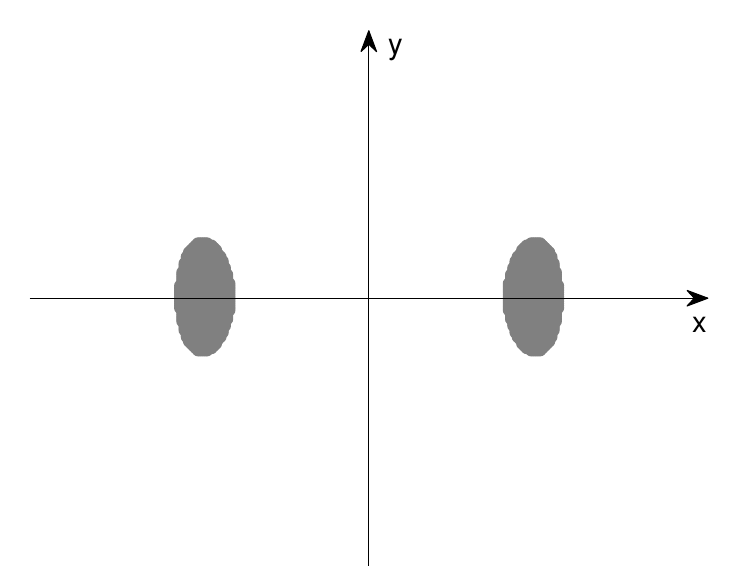}
 \caption{ \label{fig:twoballs} Representation of the thin space in the counter-example constructed in Lemma~\ref{lem:twoballs} for $n = 2$, where the greyed area corresponds to the contact set $\{u = 0\}$.}
 \end{figure}

The previous is a  simple example that exploits the fact that the negativity set of a polynomial can be disconnected when the order is higher or equal than 4. Still, each of the connected components in the previous example is convex.  Let us show that this is not needed: 

\begin{lem}
\label{lem:disk}
There exists a global solution to \eqref{eq:main_TOP} with polynomial growth \eqref{eq:poly_m} and $m = 4$ whose contact set is compact and has a connected component that is nonconvex. 
\end{lem}

\begin{proof}
In the previous proof, take 
\[
\bar p(x') = 4(|x'|^2-1)^2-1 
\]
instead. By rotational symmetry, the contact set has a connected component that is an annulus around the origin.  See Figure~\ref{fig:disk}.
\end{proof}

 \begin{figure}
 \includegraphics[scale = 0.6]{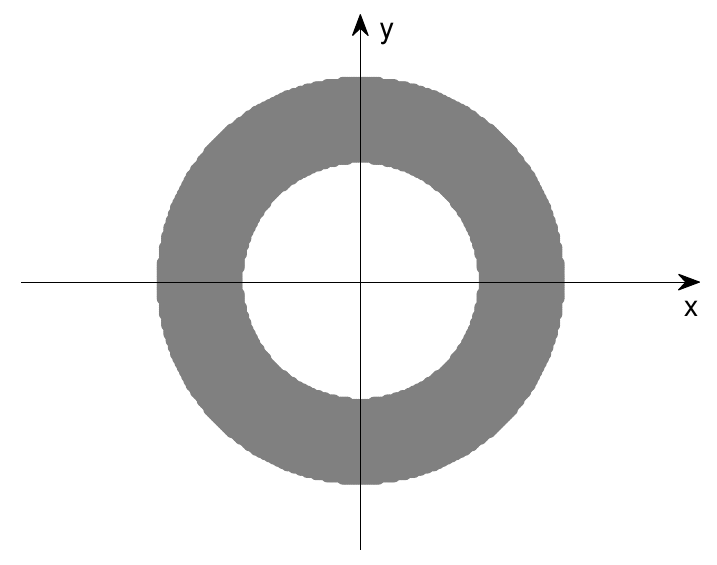}
 \caption{ \label{fig:disk} Representation of the thin space in the counter-example constructed in Lemma~\ref{lem:disk} for $n = 2$, where the greyed area corresponds to the contact set $\{u = 0\}$.}
 \end{figure}

Again, the previous example is not very satisfactory. One could conjecture that nonconvexity is achieved by either taking unions of convex components, or by removing convex sets from already convex contact sets. Let us show, by proving Lemma~\ref{lem:starshaped}, that this is not the case, and we can in fact construct counter-examples that are simply connected. 


Before constructing it, we first state and prove the following independent lemma (cf. Lemma~\ref{lem:nonconvex_starshaped2}):
\begin{lem}
\label{lem:nonconvex_starshaped}
Let $n \ge 2$. Let $p\in \mathcal{P}^e_{n+1}$ be of the form $p = p^h - 1$ for some homogeneous polynomial $p^h$ of order $m\ge 2$. Then, the solution $u$ to \eqref{eq:main_TOP} with $|u(x) - p(x) |\to 0$ as $|x|\to \infty$ has a compact and star-shaped contact set. 
\end{lem}
\begin{proof}
Let $u_r := r^{-m} u(r\, \cdot\,)$ for some $r \le 1$. Then $u_r$ is a global solution to \eqref{eq:main_TOP} such that $| u_r - p_r|\to 0$ as $|x|\to \infty$, where $p_r(x) = r^{-m}p(rx) = p^h(x) - r^{-m}\le p(x)$, since $r\le 1$. In particular, $u_r\le u$ and we have 
\[
\{u_r = 0, x_{n+1} = 0\} = \frac{1}{r}\{u = 0, x_{n+1} = 0\} = \frac{1}{r}\Lambda(u) \supset \Lambda(u)\qquad\text{for all}\quad r \le 1. 
\]
That is, $\Lambda(u)$ is star-shaped at the origin. 
\end{proof}

Lemma~\ref{lem:starshaped} follows, in fact, from the much more general statement in Proposition~\ref{prop:polysets}, which says that compact contact sets of global solutions can be arbitrarily close to zero level sets of polynomials, and we prove next, after the following preliminary lemma:

\begin{lem}
\label{lem:globk}
Let $n \ge 2$. Let $v_k$ be the global solution to \eqref{eq:TOP_-p} with $p(x) = p_k(x', x_{n+1}) =  |x'|^k-1$. Then, 
\begin{equation}
\label{eq:Bcontained}
B_{1-c_{n, k}}\cap \{x_{n+1} = 0 \}\subset \{v_k = -p_k, x_{n+1} = 0\}\subset B_{1}\cap \{x_{n+1} = 0 \},
\end{equation}
for some constant $c_{n, k}$ depending only on $n$ and $k$ such that $c_{n, k}\downarrow 0$ as $k \to \infty$. 
\end{lem}
\begin{proof}
The second inclusion in \eqref{eq:Bcontained} is clear, since $-p_k \le 0$ in $B_1^c\cap \{x_{n+1} = 0\}$. On the other hand, the first inclusion always holds for some $c_{n, k}> 0$, since the solution has rotational symmetry in the thin space and the contact set is star-shaped by Lemma~\ref{lem:nonconvex_starshaped}. Let us quantify the value of $c_{n, k}$. In order to do it, we will bound the solution $v_k$ from above by appropriate barriers. 

Let us define, for $A\ge 0$ and $t\ge 0$, 
\[
\Gamma_{A, t}(x) := A|x|^{1-n} + t, 
\]
which is harmonic outside of the origin, and nonnegative at infinity for $t\ge 0$. Then, $\Gamma_{A, t}\ge v_k$ for $t\ge 1$, and we can slide it from above by decreasing $t$ until we touch the solution $v_k$:
\[
t_* = t_*(A, k) := \min\{t > 0 : \quad \Gamma_{A, t} \ge v_k\quad\text{in}\quad \R^{n+1}\}. 
\]

Then, if $t_* > 0$, $\Gamma_{A, t_*}$ and $v_k$ are tangent at a point $x_*$ (with  $x_*= 0$ if and only if $A= 0$). By the strong maximum principle, $v_k$ cannot be harmonic around $x_*$, and  $x_*\in \{v_k = -p_k, x_{n+1} = 0\}$, that is, $x_*$ belongs to the contact set. Hence, we have that the set 
\[
S_k := \left\{x\in \R^n\times \{0\} : \begin{array}{l}\Gamma_{A, t} \ge -p_k~~\text{on}~~\{x_{n+1} = 0\}~~\text{and}~~\Gamma_{A, t}(x) = -p_k(x)\\
\text{for some}~~A\ge 0, ~t > 0.
\end{array} \right\}
\]
is contained in the contact set, 
\begin{equation}
\label{eq:Sk_contained}
S_k \subset \{v_k = -p_k, x_{n+1} = 0\}. 
\end{equation}
We immediately have that $0\in S_k$, taking $A = 0$ and $t = 1$. Moreover, we also have that $x_*\in \{x_{n+1}=0\}$ with $|x_*| = \rho>0$ is such that $x_*\in S_k$ if and only if 
\[
g(r) := Ar^{1-n}+t-1+r^k \ge 0\quad\textrm{for}\quad r > 0,\quad\text{and}\quad g(\rho) = 0,
\]
for some $A > 0$ and $t \ge 0$. The minimum of $g$, $\bar r$, is achieved when $g'(\bar r) = 0$, that is, 
\[
\bar r = \left(\frac{(n-1)A}{k}\right)^{\frac{1}{n+k-1}}.
\]
Then, we can take $\rho = \bar r$ if and only if $g(\bar r) = 0$ for some $A>0$ and $t\ge 0$. Given $A$, $t$ is immediately fixed by $g(\bar r) = 0$, and we get the condition
\[
t = 1-\bar r^k-A\bar r^{1-n} = 1- \left(\kappa^{-\frac{\kappa}{\kappa+1}}+  \kappa^{\frac{1}{\kappa+1}}\right)A^{\frac{\kappa}{\kappa + 1}}\ge 0,
\]
where $\kappa = \frac{k}{n-1}$. This gives a condition on $A$, which can be equivalently written in terms of $\bar r$ as 
\begin{equation}
\label{eq:rhonk_def}
\bar r \le \bar\rho(k, n):=  \kappa^{-\frac{1}{(n-1)(\kappa+1)}}\left(\kappa^{-\frac{\kappa}{\kappa+1}}+  \kappa^{\frac{1}{\kappa+1}}\right)^{-\frac{1}{(n-1)\kappa}} < 1,\qquad\kappa = \frac{k}{n-1}. 
\end{equation}
Thus, we have that $x_*\in S_k$ if and only if $|x_*|\le \bar \rho(k, n)$, and thanks to \eqref{eq:Sk_contained} we get that \eqref{eq:Bcontained} holds with $c_{n, k} := 1-\bar\rho(n, k)$. Finally, from \eqref{eq:rhonk_def} we get $c_{n, k}\downarrow 0$ as $k\to \infty$. 
\end{proof}

Thanks to the previous result, we get:  

\begin{proof}[Proof of Proposition~\ref{prop:polysets}]
We will construct solutions $v_p$ to \eqref{eq:TOP_-p} for some $p\in \mathcal{P}^e_{n+1}$, so that our desired solution is $u_p = v_p + p$. 

Let $v_k$ be the global solution to \eqref{eq:TOP_-p} constructed in Lemma~\ref{lem:globk}, and let $\bar \rho(n, k) = 1-c_{n, k}$. Let $w_k$ be the rescaling $w_k(x) = v_k(x/\bar\rho(k, n))$,  that is, the global solution to \eqref{eq:TOP_-p} with $p =  p_k(\cdot /\bar\rho(n, k))$, where without loss of generality we take $p_k$ to be the harmonic extension of $|x'|^k-1$ towards $\R^{n+1}$. By Lemma~\ref{lem:globk} rescaled, and taking $k$ even (so that $p$ is polynomial), we get 
\begin{equation}
\label{eq:togstep2}
B_{1}\cap \{x_{n+1} = 0 \}\subset \{w_k = -p_k(\cdot /\bar\rho(k, n)), x_{n+1} = 0\}\subset B_{1/\bar \rho(n, k)}\cap \{x_{n+1} = 0 \}.
\end{equation}
We claim that our desired solution is $u_\beta := v_{s_\beta}+s_\beta$, where $v_{s_\beta}$ is the solution to \eqref{eq:TOP_-p} with $p = s_\beta$ and 
\[
s_\beta(x) := p_k(x/\bar \rho(n, k))+\beta f_{\rm ext}(x),
\]
for some $\beta> 0$ large enough, where we have denoted by $f_{\rm ext}\in \mathcal{P}_{n+1}$ the (even) harmonic extension of $f$ towards $\R^{n+1}$. Indeed, since $s_\beta \ge p_k(\cdot/\bar\rho(n, k))$, then $w_k$ is a supersolution for the problem with obstacle $s_\beta$, and $v_{s_\beta} = -s_\beta$ in $B_1\cap \{x_{n+1} = 0\}\cap \{f_{\rm ext}(x) = 0\}$. That is, 
\[
(\{  f = 0\}\times \{0\})\cap B_1\subset \{u_\beta = 0\}\cap \{x_{n+1} = 0\}. 
\]
On the other hand, for any $\eps  >0$ there exists some $\beta > 0$ large enough so that 
\[
s_\beta(x) > 0\quad\text{in}\quad \{(x', 0) \in \R^{n+1} : {\rm dist}(x', \{f = 0\}) > \eps\}
\]
(since $ f \ge 0$), and together with  \eqref{eq:togstep2} gives, for this $\beta$,
\[
\{u_\beta = 0\}\cap \{x_{n+1} = 0\} \subset \{ (x', 0) \in \R^{n+1} : {\rm dist}(x', \{f = 0\})\le \eps,  |x'|\le 1/\bar\rho(n, k)\}.
\]
By choosing $k = 2$ we get the first result, with $R = \frac{1}{\bar\rho(n, 2)}$, which is explicit, \eqref{eq:rhonk_def}. Notice, also, that to get Remark~\ref{rem:mhom} we can take $k = m$ (which is even, since $f \ge 0$), in which case the constant $R$ also depends on $m$. 
 
 Finally, in order to get the second part of the statement, we simply observe that taking $k$ large enough we can make sure that $ \frac{1}{\bar\rho(n, k)}\le 1+\eps$. 
\end{proof}

 \begin{figure}
 \includegraphics[scale = 0.6]{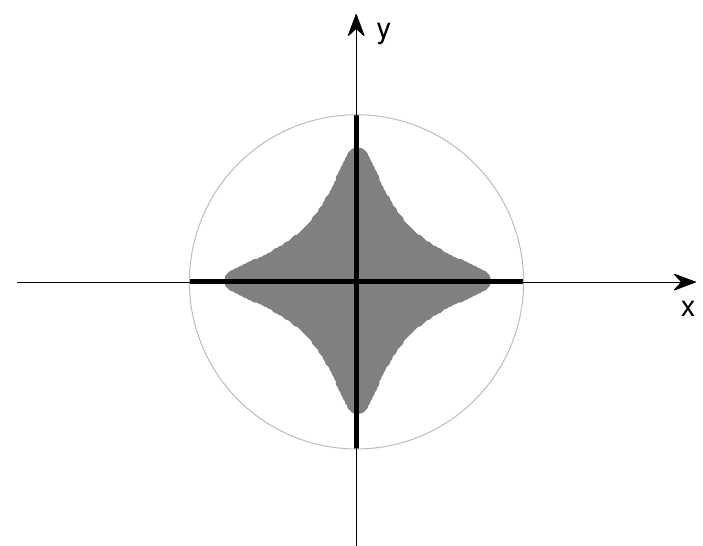}
 \caption{ \label{fig:starshaped} Representation of the thin space in the counter-example constructed in the proof Lemma~\ref{lem:starshaped} for $n = 2$, where the greyed area corresponds to the contact set $\{u = 0\}$, and the blacked lines correspond to the zero level set of $q(x, y)$ in a ball. (This is, in fact, a numerical representation in MatLab of the corresponding setting.) }
 \end{figure}
 
We can now construct the example from Lemma~\ref{lem:starshaped}, which we represent in Figure~\ref{fig:starshaped}:
\begin{proof}[Proof of Lemma~\ref{lem:starshaped}]
Take $f(x, y) = x^2y^2$ for $n = 2$ in Proposition~\ref{prop:polysets}, and we are done thanks to Remark~\ref{rem:mhom} and Lemma~\ref{lem:starshaped}. 
\end{proof}

\section{Approximating any compact set}
\label{sec:6}

We start by proving the following  abstract result on the approximation of level sets by continuous functions:

\begin{lem}
\label{lem:abs}
Let $n \ge 1$. Let $g\in C(\overline{B_1})$ with $g > 0$ on $\partial B_1$. Then, for any $\eps > 0$ there exists $\bar \eta > 0$ such that 
\[
\dH(\{g \le 0\}, \{g \le \eta_*\}) \le \eps,\qquad\text{for all}\quad \eta_*\in (0, \bar \eta],
\]
and that for any $A\subset B_1$ such that $\{g <\eta\}\subset A \subset \{g \le \bar \eta\}$ for some $\eta > 0$ we have 
\[
\dH(\partial \{g \le 0\}, \partial A) \le \eps. 
\]
\end{lem}
\begin{proof}
For the first part, we can  take, for example, 
\begin{equation}
\label{eq:eta1}
\bar \eta = \inf_{\{x'\in B_1: {\rm dist}(x', \{g\le 0\}) > \eps\}} g(x') > 0. 
\end{equation}
For the second part, observe that on the one hand, for any $x\in \partial A \subset \{\eta \le g \le \bar \eta\}$, by the first part we have 
\[
{\rm dist}(x, \partial\{g \le 0\}) = {\rm dist}(x, \{g \le 0\}) \le \dH(\{g\le \eta_*\}, \{g\le 0\}) \le \eps. 
\]

We now claim that, up to making $\bar \eta$ smaller, we can also assume that 
\begin{equation}
\label{eq:propeta}
\text{for all}\quad \bar x\in \partial\{g\le 0\}, \quad\text{ there exists }\quad \bar y \in B _\eps(\bar x )\quad\text{such that}\quad g(\bar y )\ge 2\bar\eta. 
\end{equation}
Indeed, otherwise we would have a sequence $\bar x_k\in \partial\{g\le 0\}\to \bar x_\infty\in \partial\{g \le 0\}$ (since $\partial\{g \le 0\}$ is compact)  such that $g(y )\le 0$ for all $y \in B_\eps (\bar x_\infty)$, hence $\bar x_\infty\notin\partial\{g \le 0\}$, a contradiction. 

Let  $x\in \partial\{g \le 0\}$, and let $y\in B_\eps(x)$ such that $g(y) \ge 2\bar \eta$ given by \eqref{eq:propeta}. Let $x_t := (1-t) x + t y$ for $t\in[0, 1]$. By assumption, $x_0\in A$ and $x_1\notin A$. Hence, there exists $\tau\in [0, 1]$ such that $x_\tau \in \partial A$. Since $|x_\tau - x|< \eps$, we have ${\rm dist}(x, \partial A) \le \eps$. In all, we have $\dH(\partial A, \partial\{g\le 0 \}) \le \eps$, as we wanted. 
\end{proof}

 We show next that any subzero level set of a polynomial can be approximated by global solutions with polynomial growth:

\begin{prop}
\label{prop:polysubsets}
Let $n \ge 2$. Let $f\in \mathcal{P}_n$ with $\{f < 0\}$ compact (or bounded below, for $f$ nonconstant). Then, for any $\delta > 0$ there exists a solution $u = u_\delta$ to \eqref{eq:main_TOP} with polynomial growth such that 
\[
\{f \le - \delta\}\times \{ 0 \} \subset \Lambda(u_\delta)\subset \{f \le 0\}\times\{0\}. 
\]
\end{prop}
\begin{proof}
Since $\{ f < 0\}$ is compact we have that $f$ is bounded below (they are equivalent, for $f$ nonconstant), and up to multiplying by a constant and rescaling we can assume 
\[
f \ge -1 \quad\text{in}\quad \R^n \quad\text{and}\quad \{f \le 0\}\subset B'_{1}.
\]
For each $k\in \N$, define the (single-variable) polynomial 
\[
p_{2k}(t) := 1-\left(\frac{t+1}{1-\delta}\right)^{2k}\in \mathcal{P}_1. 
\]
We know that 
\[
\frac{f+1}{1-\delta}\ge 0\quad\text{in}\quad \R^n,
\]
and 
\[
\frac{f+1}{1-\delta}\ge 1\quad\Leftrightarrow \quad f \ge -\delta. 
\]
In particular, 
\[
p_{2k}(f) \le 1 \quad\text{in}\quad \R^n, 
\]
and
\[
\begin{split}
  p_{2k}(f(x)) &\uparrow  1\quad\text{for}\quad x\in \{f < -\delta\},\quad\text{as}\quad k \to \infty,\\
p_{2k}(f(x)) &\le 0\quad\text{for}\quad x\in \{f \ge -\delta\}.
\end{split}
\]

That is, $p_{2k}(f)$ is a polynomial with compact positive set in $\R^n$, such that 
\begin{equation}
\label{eq:boundbelow}
\max\{p_{2k}(f(x)), 0\}\uparrow \chi_{\{f < -\delta\}}(x),\quad\text{for}\quad x\in \R^n,
\end{equation}
 where $\chi_A$ denotes the characteristic function of a set $A$. 
 
 We let $w_k$ be the monotonically nondecreasing sequence of  global solutions to the thin obstacle problem \eqref{eq:TOP_-p} with thin obstacle given by the polynomial  $-p = p_{2k}(f)$. Then, each $w_k$ satisfies 
 \begin{equation}\label{eq:w_bound0}
 0 \le w_k\le 1 \quad\text{in}\quad \R^{n+1}\quad\text{for all}\quad k \in \N,
\end{equation}
 and 
 \[
 \Lambda(w_k) \subset \{p_{2k}(f) > 0\}\times \{0\} \subset \{f < -\delta\}\times\{0\} \subset \{f \le 0\}\times\{0\}. 
 \]
 They also satisfy, by comparing them with the fundamental solution (thanks to Lemma~\ref{lem:comp}, since $p_{2k}(f) \le 1$ in $\R^n$ and $p_{2k}(f) \le 0$ in $\R^n\setminus B_1'$), 
 \begin{equation}\label{eq:w_bound}
 w_k(x) \le |x|^{1-n}\quad\text{for}\quad x\in \R^{n+1},
 \end{equation}
so that   they all vanish uniformly to zero at infinity.

 Notice that $p_{2k}(f)$ is uniformly $C^{1,1}_{\rm loc}$ in $\{f < -\delta\}$. More precisely,  we have 
 \[
 \lim_{k \to \infty} \|p_{2k}(f) -1\|_{C^{2}(K')} =  \lim_{k \to \infty} \left\|\left(\frac{f+1}{1-\delta}\right)^{2k}\right\|_{C^{2}(K')} =  0
 \]
 for any $K'\Subset \{f < -\delta\} = \left\{\frac{f+1}{1-\delta} < 1\right\}$. In particular, by classical $C^{1,\alpha}$ estimates for the thin obstacle problem, \cite{Caf79}, (together with local harmonic estimates), $w_k$ converges in $C^{1}_{\rm loc}$ in $\left(\{f < -\delta\}\times\{0\}\right)\cup\{x_{n+1} \neq 0\}$ to some $w_\infty$ (up to subsequences), which also satisfies \eqref{eq:w_bound0}-\eqref{eq:w_bound}, and is harmonic in $\{x_{n+1}\neq 0\}$. Also, from \eqref{eq:boundbelow} we have $w_\infty \equiv 1$ in $\{f< -\delta\}\times \{0\}$ (and $w_\infty \not\equiv 1$ by \eqref{eq:w_bound}), and by Hopf's lemma $\partial_{n+1} w_\infty < 0$ in $\{f < -\delta\}\times\{0\}$. 
 
Since $f$ is a polynomial and $\{f \le 0\}$ is compact, we know 
 \[
 \{ f < -2\delta \}\Subset \{f < -\delta\},
 \]
 and so, the previous Hopf's lemma actually implies that there exists $c_\delta > 0$ with 
 \[
 \partial_{n+1} w_\infty < -c_\delta < 0\quad\text{in}\quad \{f < -2\delta\}.
 \]
From the local uniform convergence in $C^{1}_{\rm loc}\left(\{f < -\delta\}\times\{0\}\right)$, we obtain that for all $k$ large enough, 
  \[
 \partial_{n+1} w_k< -c_\delta/2 < 0\quad\text{in}\quad \{f < -2\delta\}.
 \]
 That is, $\{f < -2\delta\}\times\{0\}\subset \Lambda(w_k)$ if $k$ is large enough. By renaming $2\delta$ into $\delta$, we get the desired result with $u_\delta = w_k+\tilde p$ and a sufficiently large $k$, where $\tilde p$ denotes the harmonic extension of $-p_{2k}(f)$ from $\{x_{n+1} = 0\}$ towards $\{x_{n+1}\neq 0\}$. 
\end{proof}

As a consequence of the previous results, we obtain the proof of Theorem~\ref{thm:main2}:

\begin{proof}[Proof of Theorem~\ref{thm:main2}]
Up to a rescaling, we assume $K \subset B_{1/4}$, and denote $K = K'\times\{0\}$. Let $d(x') = {\rm dist}(x', K'):\R^n \to \R$, which is a 1-Lipschitz continuous function. Let $\bar\eta$  be given by Lemma~\ref{lem:abs}, which depends only on $\eps > 0$ and $K'$ (through the continuous function $g(x') = d(x')$ in $\R^n$). 

Let $p\in \mathcal{P}_n$ be a polynomial such that 
\[
|p(x') - d(x')|\le \bar\eta/4,\quad\text{for}\quad x'\in B'_1. 
\]
Let 
\[
\mathcal{P}_n\ni \bar f(x') := p(x')+\left(\frac{3}{2}|x'|\right)^{2k},\quad\text{with}\quad k\in \N. 
\]
Then, for $k$ large enough, $\bar f > 0$ in $\R^n\setminus B'_1$ and 
\[
|\bar f(x') - d(x')|\le \bar\eta/3,\quad\text{for}\quad x'\in B_{1/2}. 
\]
Define, finally, 
\[
f(x') := \bar f(x')-2\bar \eta/3. 
\]
Then, we have 
\[
K' \subset \{d\le \bar\eta/3\}\subset \{f \le 0\}\subset \{d \le \bar\eta\}. 
\]
That is, we have by Lemma~\ref{lem:abs}
\[
\dH(\{f \le 0\}, K') \le \dH(\{d\le 0\}, \{d\le \bar\eta\}) \le \eps,
\]
and, taking $ A= \{f\le 0\}$ in Lemma~\ref{lem:abs}, 
\[
\dH(\partial K', \partial \{f\le 0\})\le \eps. 
\]

Finally, for any $\delta > 0$ we can find $u_\delta$ a global solution with polynomial growth, given by Proposition~\ref{prop:polysubsets} with $f-2\delta$, such that  
\[
\{f \le \delta\}\times\{0\}\subset \Lambda(u_\delta) \subset \{f \le 2\delta \}\times\{0\}. 
\]
We choose $\delta$ small enough, given by Lemma~\ref{lem:abs} with $g = f$, so that 
\[
\dH(\{f \le 0\}\times\{0\}, \Lambda(u_\delta) )\le \dH(\{f \le 0\},\{f \le 2\delta\} )\le \eps, 
\]
and 
\[
\dH(\partial\{f \le 0\}\times\{0\}, \Gamma(u_\delta))\le \eps. 
\]
Thus, we have $K\subset \{f\le 0\}\times\{0\}\subset \Lambda(u_\delta)$, and by the triangle inequality, 
\[
\dH(\Lambda(u_\delta), K) \le \dH(\{f\le 0\}, K') +\dH(\{f \le0\}\times\{0\}, \Lambda(u_\delta)) \le 2\eps,
\]
and
\[
\dH(\Gamma(u_\delta), \partial_{\R^n} K) \le \dH(\partial \{f\le 0\}, \partial K') +\dH(\partial \{f \le0\}\times\{0\}, \Gamma(u_\delta)) \le 2\eps,
\]

Up to renaming $2\eps$ into $\eps$, we get the desired result. 
\end{proof}


\end{document}